\UseRawInputEncoding
\documentclass[12pt]{article}
\usepackage[centertags]{amsmath}
\usepackage{amsfonts}
\usepackage{amssymb}
\usepackage{latexsym}
\usepackage{amsthm}
\usepackage{newlfont}
\usepackage{graphicx}
\usepackage{listings}
\usepackage{booktabs}
\usepackage{abstract}
\usepackage{enumerate}
\usepackage{xcolor}
\RequirePackage{srcltx}
\date{}

\newlength{\defbaselineskip}
\newcommand{\setlinespacing}[1]%
           {\setlength{\baselineskip}{#1 \defbaselineskip}}

\newcommand{\N}{{\mathbb{N}}}

\newcommand{\actaqed}{\hfill $\actabox$}
{\medskip\noindent \textit{Proof of #1. }}%
{\actaqed \medskip}

\def\D{{\mathcal D}}

\def\A{{\mathcal A}}
\def\cA{{\mathcal A}}
\def\C{{\mathcal C}}
\def\cC{{\mathcal C}}

\def \Tr{\mathcal T}

\def \cN{\mathcal N}

\def \cL{\mathcal L}

\def \cX{\mathcal X}

\def\R{{\mathbb R}}

\def\bbC{\mathbb C}
\def \<{\langle}
\def\>{\rangle}

\def \La{\Lambda}
\def \Og{\Omega}

\def \e{\varepsilon}
\def \va{\varepsilon}

\def \ff{\varphi}

\def \sp{\operatorname{span}}

\def\bx{\mathbf x}
\def\by{\mathbf y}

\def\bw{\mathbf w}

\def\bF{\mathbf F}

\newtheorem{Theorem}{Theorem}[section]

\newtheorem{Definition}{Definition}[section]
\newtheorem{Proposition}{Proposition}[section]
\newtheorem{Remark}{Remark}[section]

\newtheorem{Corollary}{Corollary}[section]
\numberwithin{equation}{section}

\newcommand{\be}{\begin{equation}}
\newcommand{\ee}{\end{equation}}

\def\da{{\delta}}

\def\Bl{\Bigl}
\def\Br{\Bigr}
\def\f{\frac}

\def\vi{\varphi}

\def\va{\varepsilon}

\def\CD{{\mathcal D}}

\def\CC{{\mathbb C}}

\def\NN{{\mathbb N}}

\def\Og{\Omega}

\def\sub{\substack}

\def\cN{\mathcal{N}}
\def\cA{\mathcal{A}}

\def\cX{\mathcal{X}}

\def\spn{\operatorname{span}}
\def\bx{\mathbf{x}}

\def\by{\mathbf{y}}

\def\bw{\mathbf{w}}

\def\bx{\mathbf{x}}

\DeclareFontEncoding{FMS}{}{}
\DeclareFontSubstitution{FMS}{futm}{m}{n}
\DeclareFontEncoding{FMX}{}{}
\DeclareFontSubstitution{FMX}{futm}{m}{n}
\DeclareSymbolFont{fouriersymbols}{FMS}{futm}{m}{n}
\DeclareSymbolFont{fourierlargesymbols}{FMX}{futm}{m}{n}
\DeclareMathDelimiter{\VT}{\mathord}{fouriersymbols}{152}{fourierlargesymbols}{147}

\begin{document}

\title{Lebesgue-type inequalities in sparse sampling recovery}

\author{ F. Dai and   V. Temlyakov 	\footnote{
		The first named author's research was partially supported by NSERC of Canada Discovery Grant
		RGPIN-2020-03909.
		The second named author's research was supported by the Russian Science Foundation (project No. 23-71-30001)
at the Lomonosov Moscow State University.
  }}

\newcommand{\Addresses}{{% additional braces for segregating \footnotesize
  \bigskip
  \footnotesize

  F.~Dai, \textsc{ Department of Mathematical and Statistical Sciences\\
University of Alberta\\ Edmonton, Alberta T6G 2G1, Canada\\
E-mail:} \texttt{fdai@ualberta.ca }

 \medskip
  V.N. Temlyakov, \textsc{ Steklov Mathematical Institute of Russian Academy of Sciences, Moscow, Russia;\\ Lomonosov Moscow State University; \\ Moscow Center of Fundamental and Applied Mathematics; \\ University of South Carolina.
  \\
E-mail:} \texttt{temlyakovv@gmail.com}

}}
\maketitle

\begin{abstract}{
		Recently,  it has been discovered  that results on universal sampling discretization of the square norm are useful in sparse sampling recovery with error being  measured in the square norm. 
It was established that a simple greedy type algorithm -- Weak Orthogonal Matching Pursuit -- based on good points for universal 
discretization provides effective recovery in the square norm.
In this paper we extend those results by replacing the square norm with other integral norms. In this case we need to conduct our analysis in a Banach space rather than in
 a Hilbert space, making the techniques more involved. In particular, we establish that a   greedy type algorithm -- Weak Chebyshev Greedy Algorithm -- based on good points for the $L_p$-universal 
discretization provides good recovery in the $L_p$ norm for $2\le p<\infty$.
Furthermore,  we discuss the problem of stable recovery and demonstrate its close relationship with sampling discretization. 
}\end{abstract}


{\it MSC classification 2000:} Primary 65J05; Secondary 42A05, 65D30, 41A63.

\section{Introduction}
\label{I}

The theory of Lebesgue-type inequalities has been extensively studied and developed in the context of greedy algorithms 
(see \cite{VTbookMA}, Ch. 8), particularly, for the Thresholding Greedy Algorithm with respect to bases (see \cite{VTbook} and \cite{VTbookB}). 
The main purpose of this paper is to 
 establish  some Lebesgue-type inequalities 
for algorithms that are  based on function evaluations.   Notably, recent advances  in universal discretization of integral norms achieved in \cite{DT, DTM2} prove to be 
 very useful  in this context.

 We begin with a brief 
description of  some  necessary concepts on 
sparse approximation.   Let $X$ be a Banach space with norm $\|\cdot\|:=\|\cdot\|_X$, and let $\D=\{g_i\}_{i=1}^\infty $ be a given (countable)  system of elements in $X$. Given a finite subset $J\subset \NN$, we define $V_J(\D):=\spn\{g_j:\  \ j\in J\}$. 
For a positive  integer $ v$, we denote by $\mathcal{X}_v(\D)$ the collection of all linear spaces $V_J(\D)$  with   $|J|=v$, and 
denote by $\Sigma_v(\D)$  the set of all $v$-term approximants with respect to $\D$; that is, 
$
\Sigma_v(\D):= \bigcup_{V\in\cX_v(\D)} V.
$
%Set   $$ \Sigma_v^X(\D) := \{f\in \Sigma_v(\D):\,\|f\|_X\le 1\},\  \ v=1,2,\ldots.$$
Given $f\in X$,  we define
$$
\sigma_v(f,\D)_X := \inf_{g\in\Sigma_v(\D)}\|f-g\|_X,\  \ v=1,2,\cdots.
$$ 
Moreover,   for a function class $\bF\subset X$, we define 
$$
 \sigma_v(\bF,\D)_X := \sup_{f\in\bF} \sigma_v(f,\D)_X,\quad  \sigma_0(\bF,\D)_X := \sup_{f\in\bF} \|f\|_X.
 $$

We are interested in the following problem on sparse sampling recovery. \\

{\bf Problem.} How to design a practical algorithm that gives  a  sparse sampling recovery approximant  with an error comparable to the best $v$-term approximation? \\

To answer this problem, we need to introduce some  definitions  from the theory of the Lebesgue-type inequalities for greedy algorithms  (see  \cite[Section 8.7]{VTbookMA}).
 In a general setting,  we consider an algorithm (i.e., an approximation method) $\A: = \{A_v(\cdot,\D)\}_{v=1}^\infty$ with respect to a given system $\D\subset X$, which is  a  sequence 
of mappings $A_v(\cdot,\D): X\to \Sigma_v(\CD)$, $v=1,2,\cdots$.  Clearly,  
$$
\|f-A_v(f,\D)\|_X \ge \sigma_v(f,\D)_X,\   \   \ v=1,2,\cdots,\   \ f\in X.$$
We are interested in those pairs $(\D,\A)$ for which the algorithm $\A$ provides approximation  that is close to the best $v$-term approximation.
To be more precise, let  $Y$ be a Banach space such that $Y\subset X$ and $\|\cdot\|_X \le \|\cdot\|_Y$.
The following two  definitions can be found in \cite[Section 8.7]{VTbookMA} for  the case of $X=Y$.

\begin{Definition}\label{GD1} Given a positive integer $u$, we say that a countable system $\D\subset Y$ is an almost greedy system of depth $u$ with respect to an algorithm  $\A = \{A_v(\cdot,\D)\}_{v=1}^\infty$ for the pair $(X,Y)$ of Banach spaces  if  
\begin{equation}\label{AG}
\|f-A_{C_1v}(f,\D)\|_X \le C_2\sigma_v(f,\D)_Y,\quad v=1,\dots,u,\  \  \forall f\in Y
\end{equation}
for some   constants $C_1\in\NN$ and   $C_2>0$.

In the case $C_1=1$ we call it greedy instead of almost greedy.
\end{Definition}
%If $X=Y$ and $\D$ is an almost greedy system with respect to $\A$, then $\A$ provides almost ideal sparse approximation. It provides $C_1v$-term approximant as good (up to a constant $C_2$) as the ideal $v$-term approximant for every $f\in X$. In the case of $C_1=1$,  we call $\D$ a greedy system.

More generally, we have

\begin{Definition}\label{GD2}  Let   $\mathbf{a}=\{a(j)\} _{j=1}^\infty$ be a given  sequence of positive integers. 
	We say that a system $\D\subset Y$ is an $\mathbf{a}$-greedy system of depth $u\in\NN$ with respect to an algorithm  $\A = \{A_v(\cdot,\D)\}_{v=1}^\infty$  for the pair $(X,Y)$ if 
\begin{equation}\label{aG2}
\|f-A_{v'v}(f,\D)\|_X \le C_3\sigma_v(f,\D)_Y,\quad v=1,\dots,u,\  \ \forall f\in Y
\end{equation}
for some constant $C_3> 0$. 
\end{Definition}

Inequalities (\ref{AG}) and (\ref{aG2}) are called the Lebesgue-type inequalities for the algorithm $\A$.

%
%
%
%In this paper we only consider the case of the $X=L_p$ space when $t_k=t\in(0,1]$, $k=1,2,\dots$.
%When $X$ is a Banach space and the modulus of smoothness of $X$ is defined as follows
%\begin{equation}\label{ub1}
%\rho(u):=\rho(X,u):=\frac{1}{2}\sup_{x,y;\|x\|= \|y\|=1}\left|\|x+uy\|+\|x-uy\|-2\right|,
%\end{equation}
%then the uniformly smooth Banach space is the one with $\rho(u)/u\to 0$ when $u\to 0$.
%It is known that the $L_p$ space with $2\le p<\infty$ is a uniformly smooth Banach space with 
%\be\label{ub2}
%\rho(L_p,u)\le (p-1)u^2/2.
%\ee

In this paper $X$ and $Y$ will be the Lebesgue $L_p$ spaces. Let $\Omega$ be a compact subset of $\R^d$ with a probability measure $\mu$. By $L_p$ norm of a complex-valued function defined on $\Omega$ for $1\leq p<\infty$,  we understand
$$
\|f\|_p:=\|f\|_{L_p(\Omega,\mu)} := \left(\int_\Omega |f|^pd\mu\right)^{1/p}.
$$
By $L_\infty$ norm we understand the uniform norm of continuous functions
$$
\|f\|_\infty := \max_{\bx\in\Omega} |f(\bx)|
$$
and with a little abuse of notations we sometimes write $L_\infty(\Omega)$ for the space $\C(\Omega)$ of continuous functions on $\Omega$.

We now define three algorithms, which will be studied in this paper. Let $X_N$ be an $N$-dimensional subspace of the space of continuous functions $\C(\Omega)$. For a fixed $m\in\NN$ and a set of points  $\xi:=\{\xi^\nu\}_{\nu=1}^m\subset \Omega$ we associate with a function $f\in \C(\Omega)$ a vector (sample vector)
$$
S(f,\xi) := (f(\xi^1),\dots,f(\xi^m)) \in \bbC^m.
$$
Denote
$$
\|S(f,\xi)\|_p:=\|S(f,\xi)\|_{\ell_p^m}:= \left(\frac{1}{m}\sum_{\nu=1}^m |f(\xi^\nu)|^p\right)^{1/p},\quad 1\le p<\infty,
$$
and 
$$
\|S(f,\xi)\|_\infty := \max_{\nu}|f(\xi^\nu)|.
$$
Consider the following well known recovery operator (algorithm) (see, for instance, \cite{CM})
$$
\ell p(\xi)(f) := \ell p(\xi,X_N)(f):=\text{arg}\min_{u\in X_N} \|S(f-u,\xi)\|_{p}.
$$

Let $\CD=\D_N=\{\vi_i\}_{i=1}^N$ be a given  finite system  in the space $L_p(\Og,\mu)$ for some $2\leq p<\infty$.  Now we can formulate the three algorithms that  will be studied in this paper as follows.\\

{\bf Algorithm 1.} For a given system $\D_N$ and a set of points  $\xi:=\{\xi^\nu\}_{\nu=1}^m\subset \Omega$ define the algorithm 
$$
L(\xi,f) := \text{arg}\min_{L\in \cX_v(\D_N)}\|f-\ell p(\xi,L)(f)\|_p,
$$
\be\label{Alg1}
  \ell p(\xi,\cX_v(\D_N))(f):= \ell p(\xi,L(\xi,f))(f).
\ee

{\bf Algorithm 2.} For a given system $\D_N$ and a set of points  $\xi:=\{\xi^\nu\}_{\nu=1}^m\subset \Omega$ define the algorithm 
$$
L^s(\xi,f) := \text{arg}\min_{L\in \cX_v(\D_N)}\|S(f-\ell p(\xi,L)(f),\xi)\|_p,
$$
\be\label{Alg2}
  \ell p^s(\xi,\cX_v(\D_N))(f):= \ell p(\xi,L^s(\xi,f))(f).
\ee
Index $s$ stands here for {\it sample} to stress that this algorithm only uses the sample vector $S(f,\xi)$.
Clearly, $\ell p^s(\xi,\cX_v(\D_N))(f)$ is the best $v$-term approximation of $f$ with respect to $\D_N$ in 
the space $L_p(\xi) := L_p(\xi,\mu_m)$, where $\mu_m(\xi^\nu) =1/m$, $\nu=1,\dots,m$. 
To stress this fact we use the notation $B_v(f,\D_N,L_p(\xi)) := \ell p^s(\xi,\cX_v(\D_N))(f)$.

{\bf Algorithm 3.} This algorithm is a well known greedy algorithm -- the Weak Chebyshev Greedy Algorithm (the precise definition  will be given below). We apply this algorithm
in the space $L_p(\xi,\mu_m)$, which means that we only use $S(f,\xi)$ and the restriction on the set  $\xi$ of the system $\D_N$.

Algorithm 3 is the best from the point of view of practical realization. At each iteration this greedy algorithm searches over at most $N$ dictionary elements for choosing a new one and performs the $\ell_p$ projections on the appropriate subspace (alike the other two algorithms). On the other hand, the Algorithms 1 and 2 perform  $\binom{N}{v}$
iterations of the $\ell_p$ projections on the $v$-dimensional subspaces. Note that Algorithm 2 only uses
the function values at points $\xi$ and Algorithm 1 uses an extra information for choosing the $L(\xi,f)$.

Now we give the precise definition of  the Weak Chebyshev Greedy Algorithm (WCGA) in a Banach space,  which  was introduced in \cite{T1}  as a generalization  of the Weak Orthogonal Matching Pursuit (WOMP).
To be more precise, 
let $X^\ast$ denote the  dual of the Banach space $X$. 
For a nonzero element $g\in X$,  we denote by  $F_g $  a norming (peak) functional for $g$,  that is,  an element  $F_g\in X^\ast$ satisfying 
$$
\|F_g\|_{X^*} =1,\qquad F_g(g) =\|g\|_X.
$$
The existence of such a functional is guaranteed by the Hahn-Banach theorem.

Now we can define the WCGA as follows.\\

{\bf Weak Chebyshev Greedy Algorithm (WCGA).}  Let
$\tau := \{t_k\}_{k=1}^\infty$ be a given weakness sequence of  positive numbers $\leq 1$. Let $\D=\{g\}\subset X$ be a system of nonzero elements in $X$ such that $\|g\|\le 1$ for $g\in\D$. 
Given  $f_0\in X$, we define  the elements  $f_m\in X$ and $\psi_m\in \CD$  for $m=1,2,\cdots$ inductively   as follows:

\begin{enumerate}[\rm (1)]

	\item  $\psi_m  \in \D$ is any element satisfying
	$$
	|F_{f_{m-1}}(\psi_m)| \ge t_m\sup_{g\in\D}  | F_{f_{m-1}}(g )|.
	$$
	
	\item  Define
	$$
	\Psi(m) := \sp \{\psi_1,\cdots, \psi_m\},
	$$
	and let $G_m  := G_m(f_0, \CD)_X$  be the best approximant to $f_0$ from  the space $\Psi(m)$; that is, 
	$$G_m :=\underset {G\in \Psi(m)} {\operatorname{argmin}}\|f_0-G\|_X.$$
	
	\item  Define
	$$
	f_m := f_0-G_m.
	$$
	
\end{enumerate}

In this paper we  shall only consider the WCGA for  the case when $t_k=t\in (0, 1]$ for $k=1,2,\dots$. 
We also point out that in the case when $X$ is a Hilbert space,  the WCGA coincides with the well known  WOMP, which  is very popular in signal processing, and in particular, in compressed sensing. In approximation theory the WOMP is also called the Weak Orthogonal Greedy Algorithm (WOGA).

The main goal of this paper is to prove the Lebesgue-type inequalities for Algorithms 1--3. We shall prove these inequalities under certain conditions on the system $\D_N$. 

  \begin{Definition}\label{ID1} Let $1\le p <\infty$. We say that a set $\xi:= \{\xi^j\}_{j=1}^m \subset \Omega $ provides {\it one-sided  $L_p$-universal discretization}   for a collection $\cX:= \{X(n)\}_{n=1}^k$ of finite-dimensional  linear subspaces $X(n)\subset \mathcal{C}(\Og)$ if there exists a constant $D\ge 1$ such that 
 \be\label{I3}
 \|f\|_p \le D\left(\frac{1}{m} \sum_{j=1}^m |f(\xi^j)|^p\right)^{1/p} \quad \text{for any}\quad f\in \bigcup_{n=1}^k X(n) .
\ee
We denote by $m(\cX,p,D)$ the minimal $m\in\NN$ such that there exists a set $\xi$ of $m$ points, which
provides the one-sided $L_p$-universal discretization (\ref{I3}) for the collection $\cX$. 
\end{Definition}

We prove in Section \ref{ub} (see Theorem \ref{ubT3}, inequality (\ref{I6})) that every  system $\D_N$ that  provides  the one-sided $L_p$-universal discretization  (\ref{I3}) for the collection $\cX_u(\D_N)$ is a greedy system of depth $u$ with respect to Algorithm 1 for the pair $(L_p,L_\infty)$. Also, in Section \ref{ub} (see Theorem \ref{ubT5}, inequality (\ref{ub17})) we prove that every system $\D_N$ that  provides the one-sided $L_p$-universal discretization  (\ref{I3}) for the collection $\cX_{2u}(\D_N)$ is a greedy system of depth $u$ with respect to Algorithm 2 for the pair $(L_p,L_\infty)$. These are relatively easy results. The most difficult and interesting result of this paper is about Algorithm 3. In Section \ref{ub} we prove the corresponding result (see Theorem  \ref{IT1} below) under the following conditions on the system $\CD_N$. We assume  that 
$ \D_N:=\{\ff_j\}_{j=1}^N$ is a system of $N$  uniformly bounded functions on $\Og \subset \R^d$ such that
\be \label{I1}
\sup_{\bx\in\Og} |\vi_j(\bx)|\leq 1,\   \ 1\leq j\leq N,
\ee
and  there exists a constant $K>0$ such that   for any $(a_1,\cdots, a_N) \in\bbC^N,$
\begin{equation}\label{Bessel}
  \sum_{j=1}^N |a_j|^2 \le K  \left\|\sum_{j=1}^N a_j\ff_j\right\|^2_2 .
\end{equation}

The following Theorem \ref{IT1} is a conditional result, assuring  that the  Weak Chebyshev Greedy Algorithm (WCGA)  
provides good sparse recovery in the $L_p$ norm by using points of good $L_p$-universal discretization. Theorem \ref{IT2} below, proved in \cite{DT}, guarantees the existence of such points.   

Instead of the space $L_p(\Omega,\mu)$ we consider the space $L_p(\Omega_m,\mu_m)$
where $\Omega_m=\{\xi^\nu\}_{\nu=1}^m$ is from Definition \ref{ID1} and  $\mu_m(\xi^\nu) =1/m$, $\nu=1,\dots,m$. Let $\D_N(\Omega_m)$ be the restriction 
of $\D_N$ onto $\Omega_m$. Here and elsewhere in the paper,  we often use the notation $\Omega_m$ to denote the set 
$\xi$ in order to emphasize that the set $\xi$ plays the role of a new domain $\Omega_m$ consisting of $m$ points instead of 
the original domain $\Omega$. 

\begin{Theorem}\label{IT1}    Assume that $\D_N$  is a finite system    satisfying (\ref{I1}) and \eqref{Bessel} for some constant $K\ge 1$. Let  $p\in[2,\infty)$, $t\in (0,1]$ and  $D\ge 1$ be given parameters, and let  $V: = DK^{1/2}$.
Assume that $\xi=\{\xi^1,\cdots, \xi^m\}\subset \Og $  is a set  of $m$ points
in $\Og$ that provides the  one-sided  $L_p$-universal discretization (\ref{I3}) with the given constant $D$ for the collection 
$\cX_u(\D_N)$ and for some integers $1\leq u\leq N$ and  $m\ge m(\cX_u(\D_N),p,D)$.
Then 	 there exists a constant integer  $c=C(t,p)\ge 1$ depending only on $t$ and $p$ such that for any positive integer $v$  with  $(1+cV^2(\ln (Vv))) v\leq u$,  and for any given $f_0\in \cC(\Omega)$,  the WCGA with weakness parameter $t$ applied to $f_0$ with respect to the system  $\D_N(\Omega_m)$ in the space $L_p(\Omega_m,\mu_m)$ provides
\be\label{mp}
\|f_{c V^2(\ln (Vv)) v}\|_{L_p(\Omega_m,\mu_m)} \le C\sigma_v(f_0,\D_N(\Omega_m))_{L_p(\Omega_m,\mu_m)}, 
\ee
and
\be\label{mp2}
\|f_{c V^2(\ln (Vv)) v}\|_{L_p(\Omega,\mu)} \le CD\sigma_v(f_0,\D_N)_\infty,
\ee
where $C\ge 1$ is an absolute constant.
 \end{Theorem}
 
 Theorem \ref{IT1} is a conditional result. It provides the Lebesgue-type inequalities (\ref{mp}) and (\ref{mp2}) under the condition that the set $\xi=\{\xi^\nu\}_{\nu=1}^m \subset \Omega$ of $m$ points
  provides the one-sided  $L_p$-universal discretization (\ref{I3}) for the collection $\cX_u(\D_N)$. We now formulate a known result from \cite{DT}, which established existence of good points for universal discretization.  We now proceed to a special case when   $ \D_N:=\{\ff_j\}_{j=1}^N$ is a uniformly bounded Riesz system. Namely, we assume (\ref{I1}) and instead of (\ref{Bessel}) we assume that
   for any $(a_1,\cdots, a_N) \in\bbC^N,$
\begin{equation}\label{Riesz}
R_1 \left( \sum_{j=1}^N |a_j|^2\right)^{1/2} \le \left\|\sum_{j=1}^N a_j\ff_j\right\|_2 \le R_2 \left( \sum_{j=1}^N |a_j|^2\right)^{1/2},
\end{equation}
where $0< R_1 \le R_2 <\infty$.  

	\begin{Theorem}\label{IT2}   Assume that $\D_N$ is a uniformly bounded Riesz system  satisfying (\ref{I1}) and \eqref{Riesz} for some constants $0<R_1\leq R_2<\infty$.
		Let $2<p<\infty$ and let  $1\leq v\leq N$ be an integer. 		
		 Then for a large enough constant $C=C(p,R_1,R_2)$ and any $\va\in (0, 1)$,   there exist 
		$m$ points  $\xi^1,\cdots, \xi^m\in  \Og$  with 
				\begin{equation*}
		m\leq C\va^{-7}       v^{p/2} (\log N)^2,
		\end{equation*}
			such that for any $f\in  \Sigma_v(\D_N)$, 
		\[ (1-\va) \|f\|_p^p \leq \frac   1m \sum_{j=1}^m |f(\xi^j)|^p\leq (1+\va) \|f\|_p^p. \]	
	\end{Theorem}
	
In the Problem formulated above we want to build a {\it practical} algorithm for sparse sampling recovery. 
An obvious necessary condition for a practical algorithm is its stability. We discuss stability property in Section \ref{lb}. Here, we only formulate a simple remark on stability of Algorithm 3.

\begin{Remark}\label{IR1} Set $v':=c V^2(\ln (Vv)) v$. By the definition of the WCGA we have 
$$
\|f_{v'}\|_{L_p(\Omega_m,\mu_m)} \le \|f_0\|_{L_p(\Omega_m,\mu_m)} \quad \text{and} \quad \|G_{v'}\|_{L_p(\Omega_m,\mu_m)} \le 2\|f_0\|_{L_p(\Omega_m,\mu_m)}.
$$
By discretization assumption we obtain
$$
\|G_{v'}\|_{L_p(\Omega,\mu)} \le D \|G_{v'}\|_{L_p(\Omega_m,\mu_m)} \le 2D\|S(f_0,\xi)\|_p.
$$
This means that the recovery operator (nonlinear) $G_{v'} : \ell_p^m \to L_p(\Omega,\mu)$ is stable. 
\end{Remark}

We now discuss an application of Theorems \ref{IT1} and \ref{IT2} to the optimal sampling recovery. The reader can find further discussions in Section \ref{D}. For a function class $\bF\subset \cC(\Omega)$,  we  define
$$
\varrho_m^o(\bF,L_p) := \inf_{\xi } \inf_{\Psi} \sup_{f\in \bF}\|f-\Psi(f(\xi^1),\dots,f(\xi^m))\|_p,
$$
where $\Psi$ ranges over all  mappings $\Psi : \bbC^m \to   L_p(\Omega,\mu)$  and
$\xi$ ranges over all subsets $\{\xi^1,\cdots,\xi^m\}$ of $m$ points in $\Og$. 
Here, we use the index {\it o} to mean optimality. The following Theorem \ref{IT3} is a direct corollary of 
Theorems \ref{IT1} and \ref{IT2}.

\begin{Theorem}\label{IT3} Assume that $\D_N$ is a uniformly bounded Riesz system  satisfying (\ref{I1}) and \eqref{Riesz} for some constants $0<R_1\leq R_2<\infty$.
Let $2<p<\infty$ and let  $1\leq v\leq N$ be an integer. There exist an absolute constant $C$ and a constant $C(p,R_1,R_2)$
such that for any compact $\bF\subset \cC(\Omega)$ we have for $m\ge C(p,R_1,R_2)(v\ln (2v))^{p/2} (\log N)^2$
\be\label{or}
\varrho_m^o(\bF,L_p) \le C\sigma_v(\bF,\D_N)_\infty,  
\ee
with recovery provided  by a simple greedy algorithm.
 \end{Theorem}
 
 The reader can find further results on optimal sampling recovery and on stable sampling recovery in 
 Theorems \ref{ubT6} and \ref{lbT2}. 
 
\section{Some upper bounds}
\label{ub}

{\bf 2.1. Recovery by a greedy algorithm.} In this subsection we use a greedy algorithm -- Weak Chebyshev Greedy Algorithm -- to prove some 
upper bounds for optimal sparse recovery in the $L_p$ norm with $2<p<\infty$, namely, we prove Theorem \ref{IT1} here. Note that the case 
$p=2$ is studied in detail in the recent paper \cite{DTM2}. 

For notational convenience we consider here a countable system $\D=\{g_i\}_{i=1}^\infty$, which is sufficient for our applications, where the system is finite. We now formulate a result from \cite{VT144} (see also \cite{VTbookMA}, Section 8.7) under  the following  assumption (which is formulated as {\bf A3} in \cite{VT144}):

 {\bf IP($v,S$). ($v,S$)-incoherence property.}  Let $X$ be a Banach space with a norm $\|\cdot\|$. We say that a system $\D=\{g_i\}_{i=1}^\infty\subset X$ has   ($v,S$)-incoherence property with parameters   $V>0$ and $r>0$ in $X$  if for any $A\subset B$ with  $|A|\le v$ and  $|B|\le S$,  and  for any $\{c_i\}_{i\in B}\subset \CC$, we have
$$
\sum_{i\in A} |c_i| \le V|A|^r\Bl\|\sum_{i\in B} c_ig_i\Br\|.
$$

Recall that  the modulus of smoothness of a Banach space  $X$ is defined as 
\begin{equation}\label{ub1}
\rho(u):=\rho(X,u):=\sup_{\sub{x,y\in X\\
		\|x\|= \|y\|=1}}\Bigg[\f {\|x+uy\|+\|x-uy\|}2 -1\Bigg],\  \ u>0,
\end{equation}
and that $X$ is called uniformly smooth  if  $\rho(u)/u\to 0$ when $u\to 0+$.
It is wellknown that the $L_p$ space with $2\le p<\infty$ is a uniformly smooth Banach space with 
\be\label{ub2}
\rho(L_p,u)\le (p-1)u^2/2,\  \ u>0.
\ee

The following Theorem \ref{ubT1} was proved in \cite{VT144} for real Banach spaces (see also \cite{VTbookMA}, Section 8.7, Theorem 8.7.17, p.431) and in \cite{DGHKT} for complex Banach spaces. Note that this theorem was proved there under condition that $\D$ is a dictionary but its proof works for a system as well. 
 \begin{Theorem}[{\cite[Theorem 2.7]{VT144}, \cite[Theorem 8.7.17]{VTbookMA}}]\label{ubT1} Let $X$ be a Banach space satisfying that  $\rho(X, u)\le \gamma u^q$, $u>0$ for some parameter $1<q\le 2$. Suppose that $\D\subset X$  is a system in $X$ with the  ($v,S$)-incoherence property for some integers $1\leq v\leq S$ and  parameters   $V>0$ and $r>0$.   Then the WCGA with weakness parameter $t$ applied to $f_0$ and the system $\CD$ provides
$$
\|f_{C(t,\gamma,q)V^{q'}\ln (Vv) v^{rq'}}\| \le C\sigma_v(f_0,\D)_X
$$
for any positive integer  $v$ satisfying $$v+C(t,\gamma,q)V^{q'}\ln (Vv) v^{rq'}\le S,$$ where $$q':=\f q{q-1},\   \ C(t,\gamma,q) = C(q)\gamma^{\frac{1}{q-1}}  t^{-q'},$$ and $C>1$ is an absolute constant. 
\end{Theorem}

%We now discuss sparse sampling recovery in the case $\D_N$ is a uniformly bounded Riesz basis satisfying (\ref{I1}) and
%\begin{equation}\label{Riesz1}
%R_1 \left( \sum_{j=1}^N |a_j|^2\right)^{1/2} \le \left\|\sum_{j=1}^N a_j\ff_j\right\|_2 \le R_2 \left( \sum_{j=1}^N |a_j|^2\right)^{1/2},
%\end{equation}
%where $0< R_1 \le R_2 <\infty$.   

Now we are in a position to prove Theorem \ref{IT1}.\\
   
\begin{proof}[Proof of Theorem \ref{IT1}] We use discretization properties 
(\ref{I3}) in order to establish property {\bf IP($v,S$)} for the system $\D:=\D_N(\Omega_m)$ being the restriction 
of $\D_N$ onto $\Omega_m$. For any $A$ and $B$ such that $A\subset B$
the inequality (\ref{Bessel}) implies
$$
\sum_{i\in A}|c_i| \le |A|^{1/2}\left(\sum_{i\in A}|c_i|^2\right)^{1/2} \le |A|^{1/2}\left(\sum_{i\in B}|c_i|^2\right)^{1/2} 
$$
$$
\le |A|^{1/2}K^{1/2} \|\sum_{i\in B} c_i\ff_i\|_2 \le  |A|^{1/2}K^{1/2} \|\sum_{i\in B} c_i\ff_i\|_p.
$$
This means that the system $\D_N$ has the {\bf IP}$(v,S)$ in the $L_p(\Omega,\mu)$ with any parameters $v\le N$ and $v\le S\le N$ and the constants $V_1 =  K^{1/2}$ and $r=1/2$. Our assumption that the set $\xi$ provides the one-sided $L_p$-universal discretization  
(\ref{I3}) implies that the discretized system $\D_N(\Omega_m)$ has the {\bf IP($u,u$)} in the $L_p(\Omega_m,\mu_m)$ with the constant $V =V_1 D$. 

We now define a system $\D_N$ in the $L_p(\Omega_m,\mu_m)$ as the  system $\D_N(\Omega_m)$. Clearly, our assumption (\ref{I1}) implies that for any $g\in  \D_N(\Omega_m)$
we have $\|g\|_{L_p(\Omega_m,\mu_m)} \le 1$. 
We plan to apply the WCGA algorithm with the weakness parameter $t$. We now set $c=C(t,\gamma,q)$ and for a given $v$  set $u=v+c V^2(\ln (Vv)) v$ with the constant $C(t,\gamma,q)$ from Theorem \ref{ubT1} and $V=DK^{1/2}$. Note that by (\ref{ub2}) we have $q=2$ and $\gamma=(p-1)/2$. 
It remains to apply Theorem \ref{ubT1}. This proves (\ref{mp}) in Theorem \ref{IT1}. 

We now derive (\ref{mp2}) from (\ref{mp}).   Clearly, 
$$
\sigma_v(f_0,\D_N(\Omega_m))_{L_p(\Omega_m,\mu_m)} \le \sigma_v(f_0,\D_N )_\infty.
$$
Let $f\in \Sigma_v(\D_N)$ be such that  $\|f_0-f\|_\infty \le 2 \sigma_v(f_0,\D_N)_\infty$. Let us set $v':= c V^2(\ln (Vv)) v$ for brevity. Then (\ref{mp}) implies 
$$
\|f - G_{v'}(f_0,\D_N(\Omega_m))\|_{L_p(\Omega_m,\mu_m)} \le \|f-f_0\|_{L_p(\Omega_m,\mu_m)} +\|f_{v'}\|_{L_p(\Omega_m,\mu_m)} 
$$
$$
\le (2+C)\sigma_v(f_0,\D_N)_\infty.
$$
Using that $f - G_{v'}(f_0,\D_N(\Omega_m)) \in \Sigma_u(\D_N)$, by discretization (\ref{I3}) we 
conclude that
\be\label{ub3}
\|f - G_{v'}(f_0,\D_N(\Omega_m))\|_{L_p(\Omega,\mu)} \le D(2+C)\sigma_v(f_0,\D_N)_\infty.
\ee
Finally,
$$
\|f_{v'}\|_{L_p(\Omega,\mu)} \le \|f-f_0\|_{L_p(\Omega,\mu)} + \|f - G_{v'}(f_0,\D_N(\Omega_m))\|_{L_p(\Omega,\mu)}.
$$
This and (\ref{ub3}) prove (\ref{mp2}).

\end{proof}

%  For $\xi=(\xi^1,\cdots,\xi^m)\in \Omega^m$ let $\mu_\xi$ denote the probability measure
%$$
%\mu_\xi := \frac{1}{2} \mu + \frac{1}{2m} \sum_{j=1}^m \delta_{\xi^j},
%$$
%where $\delta_\bx$ denotes the Dirac measure supported at a point $\bx$.

In  Theorem \ref{IT1},  the  WCGA 
provides an  error estimate  (\ref{mp})  in the discrete norm $L_p(\Omega_m,\mu_m)$. However,
a slight modification of the above  proof of Theorem \ref{IT1} also yields a similar error estimate in the norm of  $L_p(\Omega,\mu)$
rather than in the discrete norm $L_p(\Omega_m,\mu_m)$.

\begin{Corollary}\label{ssrR1}
	Under the conditions of Theorem \ref{IT1}, we have 
	\be\label{mp3}
	\|f_{cV^2 (\ln (Vv)) v} \|_{L_p(\Omega,\mu)} \le C D \sigma_v(f_0,\CD_N)_{L_p(\Og, \mu_\xi)},
	\ee
	where   $c=c(t,p)\ge 1$ is  a constant integer and
	$$
	\mu_\xi := \f {\mu+\mu_m}2=\frac{1}{2} \mu + \frac{1}{2m} \sum_{j=1}^m \delta_{\xi^j}.
	$$
\end{Corollary}
\begin{proof}
	For convenience, we will use the notation $\|\cdot\|_{L_p(\nu)}$  to denote the norm of $L_p$ defined with respect to a measure $\nu$ on $\Og$. 
	Let $g\in \Sigma_v(\CD_N)$ be such that  $\|f_0-g\|_{L_p(\mu_\xi)} \le 2 \sigma_v(f_0,\CD_N)_{L_p(\mu_\xi)}$.
	Let as above $v':=cV^2 (\ln (Vv)) v$. 
	Then \begin{align*}
	\|f_{v'} \|_{L_p(\mu)}&\leq
	2^{1/p} \|f_0 - G_{v'}(f_0,\CD_N(\Omega_m))\|_{L_p(\mu_\xi)}\\
	&\leq 2^{1/p}\|f_0-g\|_{L_p(\mu_\xi)}+2^{1/p} \|g- G_{v'}(f_0,\CD_N(\Omega_m))\|_{L_p(\mu_\xi)}\\
	&\leq 2^{1+\f1p} \sigma_v(f_0,\CD_N)_{L_p(\mu_\xi)}+ 2^{1/p}\|g- G_{v'}(f_0,\CD_N(\Omega_m))\|_{L_p(\mu_\xi)}.
	\end{align*}
	Since
	\[ g- G_{v'}(f_0,\CD_N(\Omega_m))\in \Sigma_{v+v'} (\CD_N) \subset \Sigma_u(\CD_N),\]
	it follows by the one-sided  universal discretization that 
	\begin{align*} \|g-& G_{v'}(f_0,\CD_N(\Omega_m))\|_{L_p(\mu_\xi)}\leq CD  \|g- G_{v'}(f_0,\CD_N(\Omega_m))\|_{L_p(\mu_m)}\\
	&	\leq C D\|f_0-g\|_{L_p(\mu_m)}+C D \|f_0- G_{v'}(f_0,\CD_N(\Omega_m))\|_{L_p(\mu_m)}\\
	&\leq CD \|f_0-g\|_{L_p(\mu_\xi)}+ C \|f_{v'} \|_{L_p(\mu_m)},\end{align*}
	which, by Theorem \ref{IT2}, is estimated by 
	\begin{align*}
	&\leq CD  \sigma_v(f_0,\CD_N)_{L_p(\mu_\xi)}+ CD\sigma_v(f_0,\CD_N(\Omega_m))_{L_p(\mu_m)}\leq C D \sigma_v(f_0,\CD_N)_{L_p(\mu_\xi)}.
	\end{align*}

\end{proof}

{\bf 2.2. Recovery by $\ell_p$ minimization.}  We recall some notations and formulate a conditional result from \cite{DTM1}, which is similar to the one from \cite{VT183}. 
Let $X_N$ be an $N$-dimensional subspace of the space of continuous functions $\C(\Omega)$. For a fixed $m$ and a set of points  $\xi:=\{\xi^\nu\}_{\nu=1}^m\subset \Omega$ we associate (as above) with a function $f\in \C(\Omega)$ a vector (sample vector) $S(f,\xi) := (f(\xi^1),\dots,f(\xi^m))$.
 In addition to the norms $\|S(f,\xi)\|_{p}$, $1\le p\le \infty$, defined above, we consider 
for a positive weight $\bw:=(w_1,\dots,w_m)\in \R^m$   the following norm
$$
\|S(f,\xi)\|_{p,\bw}:= \left(\sum_{\nu=1}^m w_\nu |f(\xi^\nu)|^p\right)^{1/p},\quad 1\le p<\infty.
$$
Define the best approximation of $f\in L_p(\Omega,\mu)$, $1\le p\le \infty$ by elements of $X_N$ as follows
$$
d(f,X_N)_{L_p(\Og, \mu)} := \inf_{u\in X_N} \|f-u\|_{L_p(\Og, \mu)}.
$$

 We make the following two assumptions.

{\bf A1. Discretization.} Let $1\le p\le \infty$. Suppose that $\xi:=\{\xi^1,\cdots,\xi^m\}\subset \Omega$ is such that for any 
$u\in X_N$ in the case $1\le p<\infty$ we have
$$
C_1\|u\|_{L_p(\Og,\mu)} \le \|S(u,\xi)\|_{p,\bw}  
$$
and in the case $p=\infty$ we have
$$
C_1\|u\|_\infty \le \|S(u,\xi)\|_{\infty}  
$$
with a positive constant $C_1$, which may depend on $p$. 

{\bf A2. Weights.} Suppose that there is a positive constant $C_2=C_2(p)$ such that 
$\sum_{\nu=1}^m w_\nu \le C_2$.

Consider the following well known recovery operator (algorithm) 
$$
\ell p\bw(\xi)(f) := \ell p\bw(\xi,X_N)(f):=\text{arg}\min_{u\in X_N} \|S(f-u,\xi)\|_{p,\bw}.
$$
Note that the above algorithm $\ell p\bw(\xi)$ only uses the function values $f(\xi^\nu)$, $\nu=1,\dots,m$. In the case $p=2$ it is a linear algorithm -- orthogonal projection with respect 
to the norm $\|\cdot\|_{2,\bw}$. In the case $p\neq 2$ it is not a linear algorithm.
 Under assumptions {\bf A1} and {\bf A2}, we have

\begin{Theorem}[\cite{VT183}, \cite{DTM1}]\label{ubT2} Suppose that conditions  {\bf A1} and {\bf A2} are satisfied for some $\xi\subset\Og$ and constants $C_1, C_2>0$.  Then for any $f\in \C(\Omega)$ and  $1\le p<\infty$ we have
\begin{equation}\label{A1}
\|f-\ell p\bw(\xi, X_N)(f)\|_{L_p(\Og, \mu)} \le 2^{1/p}(2C_1^{-1}C_2^{1/p} +1)d(f, X_N)_{L_p(\Og, \mu_{\bw,\xi})},
\end{equation}
where $\mu_{\bw, \xi}$ is  the probability measure  given by 
\[ \mu_{\bw, \xi}=\f12 \mu +\f 1{2\|\bw\|_1 }\sum_{j=1}^m w_j \da_{\xi^j},\   \ \|\bw\|_1=\sum_{j=1}^m w_j.\]
Also, we have
\begin{equation}\label{A2}
\|f-\ell p\bw(\xi, X_N)(f)\|_{L_p(\Og, \mu)} \le (2C_1^{-1}C_2^{1/p} +1)d(f, X_N)_{\infty}.
\end{equation}
\end{Theorem}

In this paper we focus mainly on   the special weights 
$$
\bw =\bw_m:= (1/m,\dots,1/m).
$$
   In this case the algorithm 
$\ell p\bw_m(\xi,X_N)$ is the classical $\ell_p$ minimization algorithm discussed above in the introduction. For brevity we use the notations introduced above
$$
\ell p(\xi,X_N) := \ell p\bw_m(\xi,X_N),\quad \mu_\xi := \mu_{\bw_m,\xi}.
$$
  For $\bF\subset \cC(\Omega)$ denote
$$
 \sigma_v(\bF,\D_N)_{(p,m)}:= \sup_{\xi\in\Omega^m} \sigma_v(\bF,\D_N)_{L_{p} (\Omega, \mu_{\xi})}.
$$
For brevity, in the case $X=L_p(\Omega,\mu)$ we write $\sigma_v(f,\D_N)_p$.

For any system $\D_N$ we have
$$
2^{-1/p} \sigma_{v}(\bF,\D_N)_p\leq \sigma_{v}(\bF,\D_N)_{(p,m)}\leq  \sigma_v(\bF,\D_N)_\infty. 
$$

In this paper we study the following 
  recovery characteristic -- {\it the sparse recovery characteristic} -- which was introduced in \cite{DTM1} in the case $p=2$.
$$
\varrho^{\ell p}_{m,v}(\bF,\D_N,L_p) := \inf_{\xi \in\Og^m} \sup_{f\in \bF} \min_{L\in\cX_v(\D_N)} \|f-\ell p(\xi,L)(f)\|_p.
$$
Clearly, for any integer  $m\ge 1$
$$
\varrho^{\ell p}_{m,v}(\bF, \D_N,L_p) \ge \sigma_v(\bF,\D_N)_p.
$$
The quantity $\varrho^{\ell p}_{m,v}(\bF,\D_N,L_p)$ shows how close we can get to the ideal $v$-term approximation error $\sigma_v(\bF,\D_N)_p$ by using function values at $m$ points and by applying associated $\ell_p$ minimization algorithms. 

We prove  here two conditional theorems.  

 \begin{Theorem}\label{ubT3} Let $m$, $v$, $N$ be given natural numbers such that $v\le N$.  Let $\D_N\subset \C(\Og)$ be  a system of $N$ elements. Assume that  there exists a set $\xi:= \{\xi^j\}_{j=1}^m \subset \Omega $, which provides {\it one-sided $L_p$-universal discretization}  (\ref{I3}) for the collection $\cX_v(\D_N)$. Then for   any  function $ f \in \C(\Omega)$ we have
\be\label{I5}
  \|f-\ell p(\xi,\cX_v(\D_N))(f)\|_p \le 2^{1/p}(2D +1) \sigma_v(f,\D_N)_{L_p(\Og, \mu_\xi)}
 \ee
 and
 \be\label{I6}
  \|f-\ell p(\xi,\cX_v(\D_N))(f)\|_p \le  (2D +1) \sigma_v(f,\D_N)_\infty.
 \ee
 \end{Theorem}
 
 We now formulate a direct corollary of Theorem \ref{ubT3} for function classes. Denote by
 $\cA(m,v,k,p,D)$   the family of all collections $\cX:= \{X(n)\}_{n=1}^k$ of finite-dimensional  linear subspaces $X(n)$, $\dim X(n)=v$, of the space $\C(\Omega)$ such that for each $\cX$  there exists a set $\xi:= \{\xi^j\}_{j=1}^m \subset \Omega $, which provides {\it one-sided $L_p$-universal discretization} (\ref{I3}) for  $\cX$.

 \begin{Theorem}\label{ubT4} Let $m$, $v$, $N$ be given natural numbers such that $v\le N$. Set $k=\binom{N}{v}$. Assume that a system $\D_N$ is such that $\cX_v(\D_N) \in \cA(m,v,k,p,D)$. Then for   any compact subset $\bF$ of $\C(\Omega)$,  we have
\be\label{I7}
 \varrho_{m,v}^{\ell p}(\bF,\D_N,L_p(\Omega,\mu)) \le 2^{1/p}(2D +1) \sigma_{v}(\bF,\D_N)_{(p,m)}
 \ee
 and
 \be\label{I8}
 \varrho_{m,v}^{\ell p}(\bF,\D_N,L_p(\Omega,\mu)) \le  (2D +1) \sigma_{v}(\bF,\D_N)_\infty.
 \ee
 \end{Theorem}

We now proceed to the proof of Theorem \ref{ubT3}.

\begin{proof}[Proof of Theorem \ref{ubT3}] We prove a more general statement than Theorem \ref{ubT3}.
We define a new algorithm for a collection $\cX:= \{X(n)\}_{n=1}^k$, which is Algorithm 1 in the case $\cX=\cX_v(\D_N)$:
$$
n(\xi,f) := \text{arg}\min_{ 1\le n\le k}\|f-\ell p(\xi,X(n))(f)\|_p,
$$
\be\label{I4}
  \ell p(\xi,\cX)(f):= \ell p(\xi,X(n(\xi,f)))(f).
\ee

Suppose that   a set $\xi:= \{\xi^j\}_{j=1}^m \subset \Omega $ provides one-sided $L_p$-universal discretization (\ref{I3}) for the collection $\cX$. Then condition {\bf A1} is satisfied for all $X(n)$ from the collection $\cX$ with $C_1=D^{-1}$ and  $\bw=\bw_m$. Clearly, condition {\bf A2} is satisfied with $C_2=1$. Thus, we can apply Theorem \ref{ubT2} for each 
subspace $X(n)$ with the same set of points $\xi$. It gives for all $n=1,\dots,k$, $k=\binom{N}{v}$,
\be\label{A3}
\|f-\ell p(\xi,X(n))(f)\|_p \le 2^{1/p}(2D +1)d(f, X(n))_{L_p(\Og, \mu_\xi)},
\ee
where as above 
$$
\mu_\xi=\mu_{\bw_m, \xi}=\f12 \mu+\f 1{2m}\sum_{j=1}^m\da_{\xi^j}.
$$
 Then, inequality (\ref{A3}) and the definition (\ref{I4}) imply
\be\label{A4}
 \|f-\ell p(\xi,\cX)(f)\|_p \le  2^{1/p}(2D +1)\min_{1\le n\le k}d(f, X(n))_{L_p(\Og, \mu_\xi)}.
 \ee
 This proves inequality (\ref{I5}) of Theorem \ref{ubT3} if we take $\cX = \cX_{v}(\D_N)$. Inequality (\ref{I6}) follows from (\ref{A2}).  \end{proof}

Along with the algorithm $\ell p(\xi,\cX_v(\D_N))$ consider a version of it, which only uses the function 
values at points $\xi^1,\dots,\xi^m$, namely, Algorithm 2 defined in the Introduction. For brevity denote $L_p(\xi) := L_p(\Omega_m,\mu_m)$, where $\Omega_m=\{\xi^\nu\}_{\nu=1}^m$  and  $\mu_m(\xi^\nu) =1/m$, $\nu=1,\dots,m$. Let 
$B_v(f,\D_N,L_p(\xi))$ denote the best $v$-term approximation of $f$ in the $L_p(\xi)$ norm with 
respect to the system $\D_N$. Note that $B_v(f,\D_N,L_p(\xi))$ may not be unique. Obviously,
\be\label{ub15}
\|f-B_v(f,\D_N,L_p(\xi))\|_{L_p(\xi)} = \sigma_v(f,\D_N)_{L_p(\xi)}.
\ee

We prove the following analog of Theorem \ref{ubT3} for this algorithm.

 \begin{Theorem}\label{ubT5} Let $m$, $v$, $N$ be given natural numbers such that $2v\le N$.  Let $\D_N\subset \C(\Og)$ be  a system of $N$ elements. Assume that  there exists a set $\xi:= \{\xi^j\}_{j=1}^m \subset \Omega $, which provides {\it one-sided $L_p$-universal discretization}  (\ref{I3}) for the collection $\cX_{2v}(\D_N)$. Then for   any  function $ f \in \C(\Omega)$ we have
\be\label{ub16}
  \|f-B_v(f,\D_N,L_p(\xi))\|_p \le 2^{1/p}(2D +1) \sigma_v(f,\D_N)_{L_p(\Og, \mu_\xi)}
 \ee
 and
 \be\label{ub17}
  \|f-B_v(f,\D_N,L_p(\xi))\|_p \le  (2D +1) \sigma_v(f,\D_N)_\infty.
 \ee
 \end{Theorem}
 \begin{proof} We derive (\ref{ub16}) and (\ref{ub17}) from (\ref{ub15}).  We begin with (\ref{ub17}).
 Clearly, 
$$
\sigma_v(f,\D_N)_{L_p(\xi)} \le \sigma_v(f,\D_N )_\infty.
$$
For brevity denote $u:=B_v(f,\D_N,L_p(\xi))$ and $h:=B_v(f,\D_N,L_\infty)$.   Then (\ref{ub15}) implies 
$$
\|h - u\|_{L_p(\xi)} \le \|f-h\|_{L_p(\xi)} +\|f-u\|_{L_p(\xi)} 
 \le 2\sigma_v(f,\D_N)_\infty.
$$
Using that $h - u \in \Sigma_{2v}(\D_N)$, by discretization (\ref{I3}) we 
conclude that
\be\label{ub18}
\|h - u\|_{L_p(\Omega,\mu)} \le 2D \sigma_v(f,\D_N)_\infty.
\ee
Finally,
$$
\|f-u\|_{L_p(\Omega,\mu)} \le \|f-h\|_{L_p(\Omega,\mu)} + \|h - u\|_{L_p(\Omega,\mu)}.
$$
This and (\ref{ub18}) prove (\ref{ub17}).

The proof of (\ref{ub16}) repeats the above argument and uses the following inequalities
$$
\|g\|_{L_p(\Og, \mu)}	\le 2^{1/p} \|g\|_{L_p(\Og,\mu_{\xi})},
$$
$$
\|g\|_{L_p(\Omega_m,\mu_m)}=\|g\|_{L_p(\xi)} \le 2^{1/p}\|g\|_{L_p(\Omega,\mu_\xi)}.
$$
 
 \end{proof}

   We now formulate a direct corollary of Theorem \ref{ubT5} for function classes.

 \begin{Theorem}\label{ubT6} Let $m$, $v$, $N$ be given natural numbers such that $2v\le N$. Set $k=\binom{N}{2v}$. Let $1\le p<\infty$. Assume that a system $\D_N$ is such that $\cX_{2v}(\D_N) \in \cA(m,2v,k,p,D)$. Then for   any compact subset $\bF$ of $\C(\Omega)$,  we have
\be\label{ub19}
 \varrho_{m}^{o}(\bF,L_p(\Omega,\mu)) \le 2^{1/p}(2D +1) \sigma_{v}(\bF,\D_N)_{(p,m)}
 \ee
 and
 \be\label{ub20}
 \varrho_{m}^{o}(\bF,L_p(\Omega,\mu)) \le  (2D +1) \sigma_{v}(\bF,\D_N)_\infty.
 \ee
 \end{Theorem}
 
 The following Theorem \ref{ubT7} is a direct corollary of 
Theorems \ref{ubT6} and \ref{IT2}.

\begin{Theorem}\label{ubT7} Assume that $\D_N$ is a uniformly bounded Riesz system  satisfying (\ref{I1}) and \eqref{Riesz} for some constants $0<R_1\leq R_2<\infty$.
Let $2<p<\infty$ and let  $1\leq v\leq N/2$ be an integer. There exist an absolute constant $C$ and a constant $C(p,R_1,R_2)$
such that for any compact $\bF\subset \cC(\Omega)$ we have for $m\ge C(p,R_1,R_2)v^{p/2} (\log N)^2$
\be\label{ub21}
\varrho_m^o(\bF,L_p) \le  C\sigma_{v}(\bF,\D_N)_{(p,m)} \le C\sigma_v(\bF,\D_N)_\infty.  
\ee
 \end{Theorem}
 
Theorem \ref{ubT7} gives a slightly better bound on $m$ than Theorem \ref{IT3}. However, the bound 
(\ref{or}) in Theorem \ref{IT3} is provided by a simple greedy algorithm, which performs about 
$(v\ln(2v))^{p/2}(\ln N)^2$ iterations of the $\ell_p$ projections on the $u$-dimensional subspaces with 
$u$ not exceeding $v\ln (2v)$ (in the sense of order). 
At each iteration that greedy algorithm searches over at most $N$ dictionary elements for choosing a new one. On the other hand, the algorithm $B_v(\cdot,\D_N,L_p(\xi))$ performs  $\binom{N}{v}$
  iterations of the $\ell_p$ projections on the $v$-dimensional subspaces. 
  
  We now discuss the case $1\le p\le 2$. In this case instead of Theorem \ref{IT2} we use the following 
  recent result from \cite{DTM2}. 
  
   \begin{Theorem}\label{ubT8} Let $1\le p\le 2$. Assume that $ \D_N=\{\ff_j\}_{j=1}^N\subset L_\infty(\Og)$ is a  system  satisfying  the conditions  \eqref{I1} and   \eqref{Bessel} for some constant $K\ge 1$. Let $\xi^1,\cdots, \xi^m$ be independent 
 	random points on $\Og$  that are  identically distributed  according to  $\mu$. 
 	 Then there exist constants  $C=C(p)>1$ and $c=c(p)>0$ such that 
 	  given any   integers  $1\leq v\leq N$ and 
 	 $$
 	 m \ge  C Kv \log N\cdot (\log(2Kv ))^2\cdot (\log (2Kv )+\log\log N),
 	 $$
 	 the inequalities 
 	 \begin{equation}\label{ub22}
 	 \frac{1}{2}\|f\|_p^p \le \frac{1}{m}\sum_{j=1}^m |f(\xi^j)|^p \le \frac{3}{2}\|f\|_p^p,\   \   \ \forall f\in  \Sigma_v(\D_N)
 	 \end{equation}
 hold with probability $\ge 1-2 \exp\Bl( -\f {cm}{Kv\log^2 (2Kv)}\Br)$.
\end{Theorem}
  
 The following Theorem \ref{ubT9} is a direct corollary of 
Theorems \ref{ubT6} and \ref{ubT8}.

\begin{Theorem}\label{ubT9} Let $1\le p\le 2$. Assume that $ \D_N=\{\ff_j\}_{j=1}^N\subset L_\infty(\Og)$ is a  system  satisfying  the conditions  \eqref{I1} and   \eqref{Bessel} for some constant $K\ge 1$. Let  $1\leq v\leq N/2$ be an integer. There exist an absolute constant $C$ and a constant $C(p,K)$
such that for any compact $\bF\subset \cC(\Omega)$ we have for 
$$
 m \ge  C Kv \log N\cdot (\log(2Kv ))^2\cdot (\log (2Kv )+\log\log N)
 $$
 that
\be\label{ub23}
\varrho_m^o(\bF,L_p) \le  C\sigma_{v}(\bF,\D_N)_{(p,m)} \le C\sigma_v(\bF,\D_N)_\infty.  
\ee
 \end{Theorem}

\section{Some bounds for stable recovery}
\label{lb}

In the definition of the quantity $\varrho_m^o(\bF,L_p)$
we allow any mapping $\Psi : \bbC^m \to  L_p(\Omega,\mu)$. We now consider an analog of the quantity $\varrho_m^o(\bF,L_p)$ with some restrictions imposed on the mapping $\Psi : \bbC^m \to L_p(\Omega,\mu)$. For a given set $\xi =\{\xi^\nu\}_{\nu=1}^m\subset \Omega$, and given parameters $m\in\N$ and $A>0$ define the following class of nonlinear mappings ($A$-stable mappings)
$$
\cN\cL(\xi,L_p,A) := \{\Psi\, : \, \Psi : \bbC^m \to L_p(\Omega,\mu),\, \|\Psi(a\by)\|_p = |a|\|\Psi(\by)\|_p, 
$$
$$
\|\Psi(S(f,\xi))\|_p \le A\|S(f,\xi)\|_p \}.
$$
Consider the following recovery characteristic
$$
\varrho_m^o(\bF,L_p,A) :=  \inf_{\xi } \inf_{\Psi\in \cN\cL(\xi,L_p,A)} \sup_{f\in \bF}\|f-\Psi(f(\xi^1),\dots,f(\xi^m))\|_p.
$$
This characteristic gives the minimal error that can be achieved with an $A$-stable algorithm.

{\bf Some lower bounds.} For a subspace $X_N$ denote by $X_N^p$ the unit $L_p$-ball of this subspace. We begin with a simple statement, which connects the characteristic $\varrho_m^o(X_N^p,L_p,A)$ 
with discretization of the $L_p$ norm on $X_N$. In the case $p=\infty$ Proposition \ref{lbP1} was proved in \cite{KKT}.

\begin{Proposition}\label{lbP1} Inequality $\varrho_m^o(X_N^p,L_p,A) \le 1/2$ implies that there exists a set $\xi(m)\subset \Omega$ of $m$ points such that for any $f\in X_N$ we have
$$
\|f\|_p \le 2A\|S(f,\xi(m))\|_p.
$$
\end{Proposition}
\begin{proof} Let $\e>0$ and let $\xi^\e \in \Omega^m$ and $\Psi_\e \in \cN\cL(\xi^\e,L_p,A)$ be such that for any $f\in X_N^p$ we have
\be\label{lb1}
\|f-\Psi_\e(S(f,\xi^\e))\|_p \le 1/2+\e. 
\ee
Then for $f\in X_N$ such that $\|f\|_p =1$ we obtain
$$
1=\|f\|_p = \|f-\Psi_\e(S(f,\xi^\e))+\Psi_\e(S(f,\xi^\e))\|_p \le \|\Psi_\e(S(f,\xi^\e))\|_p+1/2+\e
$$
and, therefore,
$$
(1/2-\e)\|f\|_p \le \|\Psi_\e(S(f,\xi^\e))\|_p \le A\|S(f,\xi^\e))\|_p.
$$
This and a  simple compactness argument, used for $\e\to0$, complete the proof of Proposition \ref{lbP1}.

\end{proof}

We now make a comment on stability from \cite{KKT}. It is well known (see, for instance, \cite{DPTT}, Proposition 3.1) that for any $N$-dimensional subspace $X_N\subset \cC(\Omega)$ we can find a set $\xi =\{\xi^j\}_{j=1}^N$ of $N$ points such that 
any $f\in X_N$ can be recovered exactly from the vector $S(f,\xi)$ (even by a linear algorithm). 
However, as we will explain momentarily, such a recovery algorithm may be unstable. 

First, we 
discuss the case $p=\infty$ from \cite{KKT}. The following result was obtained in \cite{VT168} (see also \cite{DPTT}). 

\begin{Theorem}[\cite{VT168},\cite{DPTT}]\label{lbT1} Let $\Lambda_N = \{k_j\}_{j=1}^N$ be a lacunary sequence: $k_1=1$, $k_{j+1} \ge bk_j$, $b>1$, $j=1,\dots,N-1$. Assume that a finite set $\xi=\{\xi^\nu\}_{\nu=1}^m\subset \mathbb T$ has
the following property
\begin{equation}\label{lb2}
\forall f\in \Tr(\Lambda_N) \qquad \|f\|_\infty \le L\max_{\nu}|f(\xi^\nu)|.
\end{equation}
Then
$$
m \ge (N/e)e^{CN/L^2}
$$
with a constant $C>0$ which may only depend on $b$.
\end{Theorem}

As an example we take $X_n := \Tr(\La_n)$, where $\La_n$ is a lacunary set from Theorem \ref{lbT1}. Suppose that 
$\varrho_m^o(X_n^\infty,L_\infty,A) \le 1/2$ with some parameters $m$ and $A$. Clearly, this assumption is much weaker than the exact recovery by stable algorithm assumption. Then by Proposition \ref{lbP1} there exists a set $\xi(m)\subset \Omega$ of $m$ points such that for any $f\in X_n$ we have
$$
\|f\|_\infty \le 2A\|S(f,\xi(m))\|_\infty.
$$
We apply Theorem \ref{lbT1} and obtain that $m\ge (n/e)e^{Cn/(2A)^2}$. This means that for a stable nontrivial approximate recovery on $X_n$ we need exponentially (in $n$) many points. 

Second, we discuss the case $2<p<\infty$. Proposition \ref{lbP1} gives the inequality
for each $f\in X_N^p$
\be\label{lb1a}
\|f\|_p^p \le (2A)^p \frac{1}{m}\sum_{j=1}^m |f(\xi^j)|^p,
\ee
which is the one-sided discretization of the $L_p$ norm on the $X_N$. For illustration we cite a known result about the lower bound on the number of points needed for good discretization of the $L_p$ norm. 

The following Proposition \ref{lbP2} is from \cite{KKLT} 
(see {\bf D.20. A Lower bound} there). 

\begin{Proposition}\label{lbP2} Let $p\in (2,\infty)$ and let a subspace $X_N \subset \cC(\Omega)$
be such that the $L_p(\Omega,\mu)$ is equivalent to the $L_2(\Omega,\mu)$. Then it is necessary 
to have at least $N^{p/2}$ (in the sense of order) points for discretization with positive weights of the 
$L_p(\Omega,\mu)$ norm on $X_N$. 
\end{Proposition}

%In the same way as above, using Proposition \ref{lbP2} instead of Theorem \ref{lbT1}, we obtain that 
%for a stable nontrivial approximate recovery in the $L_p$ norm on $X_n := \Tr(\La_n)$ we need $n^{p/2}$ points (in the sense of order). Note, that it is well known that for any $p\in (2,\infty)$ and any $f\in 
%\Tr(\La_n)$ we have $\|f\|_p \le C(p)\|f\|_2$. 

The proof of Proposition \ref{lbP1} works for the following statement as well. 

\begin{Proposition}\label{lbP3} Let $\D_N$ be a system of $N$ elements. Inequality 
$$
\varrho_m^o(\Sigma_v^p(\D_N),L_p,A) \le 1/2
$$
 implies that there exists a set $\xi(m)\subset \Omega$ of $m$ points such that for any $f\in \Sigma_v^p(\D_N)$ we have
$$
\|f\|_p \le 2A\|S(f,\xi(m))\|_p.
$$
\end{Proposition}

Proposition \ref{lbP3} means that good recovery of elements from $\Sigma_v^p(\D_N)$ implies 
one-sided $L_p$-universal discretization for the collection $\cX_v(\D_N)$. Let us compare this with 
Theorem \ref{IT1}. Remark \ref{IR1} shows that the algorithm, which provided bound (\ref{mp2}) for 
the $L_p$ error of recovery, belongs to $\cN\cL(\xi,L_p,A)$ with $A=2D$. Inequality (\ref{mp2})
was proved under assumption that the set of points $\xi$ provides one-sided $L_p$-universal discretization for the collection $\cX_u(\D_N)$. On the other hand Proposition \ref{lbP3} applied 
to $\Sigma_v^p(\D_N)$ guarantees one-sided $L_p$-universal discretization for the collection $\cX_v(\D_N)$. Parameters $u$ and $v$ are close -- $u$ is of order $v\log v$. Therefore, it seems like the condition of one-sided $L_p$-universal discretization for the collection $\cX_v(\D_N)$ is close to the necessary and sufficient condition for the Lebesgue-type inequality (\ref{mp2}). 

Let us now discuss another stable recovery algorithm -- $B_v(\cdot,\D_N,L_p(\xi))$. Under assumption that $\xi:= \{\xi^j\}_{j=1}^m \subset \Omega $ provides  one-sided $L_p$-universal discretization  (\ref{I3}) for the collection $\cX_{2v}(\D_N)$ we obtain stability 
\be\label{lb4}
\|B_v(f,\D_N,L_p(\xi))\|_p \le D\|B_v(f,\D_N,L_p(\xi))\|_{L_p(\xi)} \le 2D\|f\|_{L_p(\xi)}.
\ee
Theorem \ref{ubT5} gives the Lebesgue inequality (\ref{ub17}). On the other hand, Proposition \ref{lbP3} implies that good recovery by an $A$-stable algorithm of functions from $\Sigma_v^p(\D_N)$ guarantees that $\xi(m)$ provides one-sided $L_p$-universal discretization  (\ref{I3}) for the collection $\cX_{v}(\D_N)$ with $D=2A$. 

{\bf Some upper bounds.} We now give a comment on the upper bounds for $\varrho_m^o(\bF,L_p,A)$. Theorem \ref{ubT5} and stability property (\ref{lb4}) imply the following analog of Theorem \ref{ubT6}.

\begin{Theorem}\label{lbT2} Let $m$, $v$, $N$ be given natural numbers such that $2v\le N$. Set $k=\binom{N}{2v}$. Let $1\le p<\infty$. Assume that a dictionary $\D_N$ is such that $\cX_{2v}(\D_N) \in \cA(m,2v,k,p,D)$. Then for   any compact subset $\bF$ of $\C(\Omega)$,  we have
\be\label{lb5}
 \varrho_{m}^{o}(\bF,L_p(\Omega,\mu),2D) \le 2^{1/p}(2D +1) \sigma_{v}(\bF,\D_N)_{(p,m)}
 \ee
 and
 \be\label{lb6}
 \varrho_{m}^{o}(\bF,L_p(\Omega,\mu),2D) \le  (2D +1) \sigma_{v}(\bF,\D_N)_\infty.
 \ee
 \end{Theorem}

\section{Discussion}
\label{D}

In this paper we proved a number of the Lebesgue-type inequalities for different sampling recovery algorithms. The Lebesgue-type inequalities can be used in the following general problem of exact recovery. Let $\D_N$ be a system (dictionary) of
functions from $\cC(\Omega)$. For given $1\le p\le \infty$ and $v\le N$ we want to find a minimal number of points $\xi^1,\dots,\xi^m$ and a stable algorithm, which recovers exactly any function 
$f\in \Sigma_v(\D_N)$ by using the function values $f(\xi^1),\dots,f(\xi^m)$. Theorem \ref{IT1} provides 
the corresponding result for Algorithm 3, Theorem \ref{ubT3} -- for Algorithm 1, and Theorem \ref{ubT5} -- for Algorithm 2. 

Let us make a comment on Algorithm 1 -- the algorithm $\ell p(\xi,\cX_v(\D_N))$. We discuss 
the noiseless situation. Let $f\in \Sigma_v(\D_N)$. Then under assumption on  the one-sided $L_p$-universal discretization for $\cX_{v}(\D_N)$  
inequality (\ref{I6}) of Theorem \ref{ubT3} guarantees that $\ell p(\xi,\cX_v(\D_N))=f$. In this case under a stronger assumption that $\xi$ provides the one-sided $L_p$-universal discretization for $\cX_{2v}(\D_N)$
we can realize this algorithm in the following way. Instead of minimizing $\|f-\ell p (\xi,L)(f)\|_p$ over 
$L\in \cX_v(\D_N)$ we check 
 the equalities $\ell p (\xi,L)(f)(\xi^\nu)= f(\xi^\nu)$, $\nu =1,\dots,m$. If these equalities are satisfied  then we stop. Theorem \ref{ubT3} guarantees that for some $L$ the equalities are indeed satisfied. 
 Next, $f-\ell p (\xi,L)(f) \in \cX_{2v}(\D_N)$ and therefore our discretization assumption guarantees 
 that $\ell p (\xi,L)(f)=f$. Thus, in this case the algorithm uses only the function values $f(\xi^\nu)$, $\nu=1,\dots,m$. The disadvantage of this algorithm is that we need to check all $L\in \cX_v(\D_N)$. 
 Note that the algorithm $B_v(\cdot,\D_N,L_p(\xi))$ also uses only the function values $f(\xi^\nu)$, $\nu=1,\dots,m$ and under assumptions of Theorem \ref{ubT5} recovers $f\in \Sigma_v(\D_N)$ exactly. 
 
 Let us now discuss the WCGA from Theorem \ref{IT1}. This algorithm only uses the function values 
$f(\xi^\nu)$, $\nu=1,\dots,m$. Remark \ref{IR1} guarantees that this algorithm is stable. Inequality (\ref{mp2}) of Theorem \ref{IT1} guarantees that the WCGA 
recovers $f_0\in \Sigma_v(\D_N)$ exactly after $cV^2(\ln(Vv))v$ iterations. Moreover, it guarantees 
that even if $f_0$ is a noisy version of $f\in \Sigma_v(\D_N)$ such that $\|f_0-f\|_\infty \le \delta$
then the WCGA provides a $cV^2(\ln(Vv))v$-sparse element, which deviates from $f$ less that 
$C\delta$. 

We now discuss the case $p=2$ in Theorem \ref{IT1}. Under assumption that the system $\D_N$   satisfies (\ref{I1}) and \eqref{Bessel}   we prove in Theorem \ref{IT1} 
that for $p=2$ we have bounds (\ref{mp}) and (\ref{mp2}) after $v'$ iterations with $v'$ of the order 
$v\ln (Vv)$. In the very recent paper \cite{DTM2} we proved the following analog of Theorem \ref{IT1}.
In Theorem \ref{DT1} we impose a stronger assumption (\ref{Riesz}) instead of (\ref{Bessel}) on the system $\D_N$ and we prove the bounds (\ref{mp'}) and (\ref{mp2'}) after $v'$ iterations with $v'$ of the order $v$. We use the following definition.

  \begin{Definition}\label{DD1} We say that a set $\xi:= \{\xi^j\}_{j=1}^m \subset \Omega $ provides {\it   universal discretization}   for the collection $\cX:= \{X(n)\}_{n=1}^k$ of finite-dimensional  linear subspaces $X(n)$ if we have
 \be\label{D3}
\frac{1}{2}\|f\|_2^2 \le \frac{1}{m} \sum_{j=1}^m |f(\xi^j)|^2\le \frac{3}{2}\|f\|_2^2\quad \text{for any}\quad f\in \bigcup_{n=1}^k X(n) .
\ee
We denote by $m(\cX)$ the minimal $m$ such that there exists a set $\xi$ of $m$ points, which
provides  universal discretization (\ref{D3}) for the collection $\cX$. 
\end{Definition}

\begin{Theorem}[\cite{DTM2}]\label{DT1}    Assume that $\D_N$ is a uniformly bounded Riesz system satisfying \eqref{Riesz} for some constants $0<R_1\leq R_2<\infty$. For given parameters $t\in (0,1]$ and $R_1$, $R_2$ from above there exists a constant $c=C(t,R_1,R_2)$ with the following property. Let for integer $v$ such that $u:=(1+c) v\leq N$ and a number $m\ge m(\cX_u(\D_N))$ the set $\xi$ of $m$ points
$\xi^1,\cdots, \xi^m \in  \Omega$ provide   universal discretization for the collection 
$\cX_u(\D_N)$. Then for any $f_0\in \cC(\Omega)$ the WCGA with weakness parameter $t$ applied to $f_0$ with respect to the normalized $\D_N(\Omega_m)$ in the space $L_2(\Omega_m,\mu_m)$ provides
\be\label{mp'}
\|f_{c v}\|_{L_2(\Omega_m,\mu_m)} \le C\sigma_v(f_0,\D_N(\Omega_m))_{L_2(\Omega_m,\mu_m)}, 
\ee
and
\be\label{mp2'}
\|f_{c v}\|_{L_2(\Omega,\mu)} \le C'\sigma_v(f_0,\D_N)_\infty 
\ee
with absolute constants $C$ and $C'$.
 \end{Theorem}

  \Addresses

\end{document}